\documentclass[10pt]{amsart}

\usepackage{amsmath,amsfonts,amssymb,amsthm,epsfig,color}

\usepackage{esint}

\usepackage{color}

\usepackage{import}

\usepackage[final,allcolors=blue,colorlinks=true]{hyperref}
\usepackage[leqno]{amsmath}

\newtheorem{theorem}{Theorem}[section]
\newtheorem{lemma}[theorem]{Lemma}
\newtheorem{proposition}[theorem]{Proposition}

\theoremstyle{definition}

\newtheorem{remark}[theorem]{Remark}
\numberwithin{equation}{section}

\newcommand{\R}{\mathbb{R}}

\newcommand{\medint}{-\kern  -,395cm\int}
\newcommand{\medintdue}{-\kern  -,31cm\int}

\newcommand{\dist}{{\rm dist}}
\newcommand{\eps}{\varepsilon}
\newcommand{\vphi}{\varphi}

\newcommand{\F}{{\mathcal{F}}}
\newcommand{\Ha}{{\mathcal{H}}}

\newcommand{\E}{{\mathcal{E}}}
\newcommand{\D}{{\mathcal{D}}}

\newcommand{\diver}{\operatorname{div}}

\renewcommand{\MR}[1]{\null}

\begin{document}

\title[On the regularity of critical and minimal sets of a free interface problem]{On the regularity of critical and minimal sets of a free interface problem}

\author{Nicola Fusco}

\address{Dipartimento di Matematica e Applicazioni ``R. Cacciopoli'', Universit\`a degli Studi di Napoli ``Federico II'', Napoli, Italy}
\email{n.fusco@unina.it}

\author{Vesa Julin}
\address{Department of Mathematics and Statistics, University of Jyv\"askyl\"a, Finland}
\email{vesa.julin@jyu.fi}

\date{\today}

\begin{abstract}
We study a free interface problem of finding the optimal energy configuration for mixtures of two conducting materials with an additional  perimeter  penalization of the interface. We employ the regularity theory of linear elliptic equations to study the possible opening angles of  Taylor cones and  to give a different proof of a partial regularity result  by Fan Hua Lin \cite{Lin}.
\end{abstract}

\maketitle

\section{Introduction}

In this paper we consider the functional
\begin{equation}
\label{energy}
\F(E, v) = \gamma P(E, \Omega) +  \int_{\Omega} \sigma_E (x)|D v|^2 \, dx, 
\end{equation}
where $\gamma>0$,   $\Omega \subset \R^n$ is an open set, $v \in W^{1,2}(\Omega)$ and $P(E, \Omega)$ stands for the  perimeter of $E$ in $\Omega$. Moreover,  $\sigma_E(x)=  \beta \chi_{E} (x) + \alpha \chi_{\Omega \setminus E}(x)$, where $0 < \alpha < \beta  < \infty$ are given constants.

Given a function $u_0 \in W^{1,2}(\Omega)$ and a measurable $E \subset \Omega$, we denote by $u_E$, or simply by $u$ if no confusion arises, the corresponding elastic equilibrium, i.e., the minimizer in $W^{1,2}(\Omega)$ of the functional 
\[
  \int_{\Omega} \sigma_E (x)|D v|^2 \, dx 
\]
under the boundary condition $v = u_0$ on $\partial \Omega$. It follows that the function $u$ solves the linear equation
\begin{equation}
\label{linear_eq}
 \int_{\Omega} \langle \sigma_E D u , D \varphi \rangle \, dx = 0  \qquad \text{for every  }\, \varphi \in W_0^{1,2}( \Omega).
\end{equation}
If we denote by $u_{\beta}$ and $u_{\alpha}$ the restriction of $u$ on $E$ and  $\Omega \setminus E$, respectively, they are harmonic in their domains. Moreover, equation \eqref{linear_eq} implies  the  transmission condition 
\begin{equation}
\label{transition}
\alpha \partial_{\nu}u_{\alpha}(x)= \beta \partial_{\nu} u_{\beta}(x) \qquad \text{for all }\,  x \in \partial E \cap \Omega,
\end{equation}
where $\partial_{\nu}$ denotes the derivative of $u$ in the direction of the exterior normal to $\partial E$. Note that if $(E,u)$ is a smooth critical point of the functional \eqref{energy} the following Euler-Lagrange equation holds
\begin{equation}
\label{1st_var}
\gamma H_{\partial E} + \beta |D u_{\beta}|^2  - \alpha |D u_{\alpha}|^2 = \lambda \quad \text{on }\, \partial E \cap \Omega,
\end{equation}
where $ H_{\partial E}$ stands for the mean curvature of $\partial E$ and $\lambda$ is either zero or a Lagrange multiplier (in case of a volume constraint).

In the physical literature critical points of the functional \eqref{energy}, i.e, solutions of equations \eqref{linear_eq} and \eqref{1st_var} are studied to model the shape of liquid drops exposed to an electric or a magnetic field. In the model the set  $E$ represents a liquid drop with  dielectric permittivity $\beta$, surrounded by a  fluid with smaller permittivity $\alpha$, and $u$ stands for the electrostatic potential induced by an applied electric field. At the interface, which is assumed to be in static equilibrium, the normal component of  the electric displacement field $\sigma_E D u$ has to be continuous. This implies that $u$ has to satisfy \eqref{transition} or equivalently that it has to be  a solution of equation \eqref{linear_eq}.  On the other hand,  on the interface   the electric stress and the surface tension has to be balanced, which leads to \eqref{1st_var}. 

The occurrence of conical tips at an interface exposed to an electric field has been observed by several authors (see e.g. \cite{Z}). Moreover, theoretical investigations  (\cite{LHL}, \cite{RC}, \cite{SLB}, \cite{T}) suggest that conical critical points, the so called \emph{Taylor cones},  may only occur if the ratio $\frac{\beta}{\alpha}$ is sufficiently large and if the opening  angle is neither too small nor too close to $\pi/2$, i.e., it belongs to  a  certain range which is independent of the penalization factor $\gamma$.  

 In order to state our first result  we denote by $E_{\theta}$ the right circular cone with opening angle $\theta \in (0 , \pi/2)$  and vertex at the origin, i.e., 
\[
E_{\theta} = \big\{  (x', x_n) \in \R^{n-1} \times \R \, : \, x_n > \frac{1}{ \tan \theta} |x'|  \big\}.
\]
\begin{theorem} 
\label{mainthm2}
Let $n \geq 3$. There exist two positive constants $\delta_0= \delta_0(n,\frac{\beta}{\alpha})>0$ and $\lambda_0 = \lambda_0(n)>1$ such that, if $E_{\theta}$ is a right circular cone  satisfying \eqref{1st_var} then $\frac{\beta}{\alpha} \geq  \lambda_0$ and 
\[
 \delta_0 \leq \theta \leq   \frac{\pi}{2} -  \delta_0.
\]
\end{theorem}

As far as we know, this result is the first rigorous proof of the fact that Taylor cones may  occur only for certain angles, and  provided that the ratio $\frac{\beta}{\alpha}$ is sufficiently large.  We remark  that we are able to give explicit estimates of the constants $\delta_0$ and $ \lambda_0$. In particular,  $\delta_0$ and $ \lambda_0$  are independent of  the penalization factor $\gamma$, which  is in accordance with the observations and theoretical results reported in the physical literature. 

The starting point in the proof of Theorem \ref{mainthm2}  is  a rather simple decay estimate for the gradient of a minimizer of the Dirichlet energy  (see Proposition~\ref{decay_est}).  Roughly speaking, we prove that if $u_E$ is a solution of \eqref{linear_eq} and $x_0$ is a point in $\Omega$, where either the density of $E$ is close to $0$  or $1$, or the set $E$ is asymptotically close to a hyperplane, then for sufficiently small $\rho$ we have 
\begin{equation} 
\label{decay_result}
\int_{B_{\rho}(x_0)}|D u_E|^2 \, dx \leq C \rho^{n- \delta}
\end{equation} 
for any $\delta >0$. As a consequence of this estimate one has that if the opening angle of the Taylor cone is  not in the above range then the Dirichlet energy around the vertex decays faster than the perimeter thus leading to a contradiction to the criticality condition  \eqref{1st_var}.

In the mathematical literature people have also considered both the problem $(P)$ of minimizing \eqref{energy} under the boundary condition $u = u_0$ on  $\partial \Omega$ and  the constrained problem 
\[
\min \{\F(E, v)\, : \, v = u_0 \,\text{ on }\,\partial \Omega, \,|E|=d  \} \leqno (P_c)
\]
for some given $d < |\Omega|$. The partial regularity of minimizers of the unconstrained problem $(P)$ was proved by Fan-Hua Lin in \cite{Lin} (see also \cite{AB}, \cite{LK}). In the special case $n=2$ the result of Lin has been improved by Larsen in \cite{L1}, \cite{L2}. However, the full regularity of the free interface $\partial E$ in two dimensions still remains open. 

In the second part of the paper we revisit the proof of the partial  regularity of minimizers.
 \begin{theorem}
\label{mainthm1}
If $(E, u)$ is a minimizer of either problem $(P)$ or problem $(P_c)$, then 
\begin{enumerate}
\item[(a)] there exists a relatively open set $\Gamma \subset \partial E$ such that $\Gamma$ is a $C^{1, \sigma}$ hypersurface for all $0 < \sigma < 1/2$,
\item[(b)]  there exists $\eps>0$, depending only on $\frac{\beta}{\alpha}$ and $n$, such that $\Ha^{n-1- \eps}\left((\partial  E  \setminus \Gamma ) \cap \Omega\right)= 0$.
\end{enumerate}
\end{theorem}

The above statement slightly  generalizes the regularity result proved in \cite{Lin}, where only the  unconstrained problem $(P)$ was considered. The value of $\eps$ is greater than or equal to   $p-1$ where $2p$ is  the higher integrability exponent of $Du_E$ which is obtained  by a standard application of Gehring's lemma (see Lemma \ref{gehring}). In particular, it is independent of the penalization factor $\gamma$.  We note also that, once the $C^{1,\sigma}$ regularity of $\partial^*E\cap\Omega$ is obtained, then using \cite[Lemma 5.3]{Lin} and a standard bootstrap argument one obtains that $\partial^*E\cap\Omega$ is in fact $C^\infty$.

As in \cite{Lin} and in the proof of the regularity of minimizers of the Mumford-Shah functional \cite{AFP} our proof of Theorem \ref{mainthm1} is based on a interplay between the perimeter and the Dirichlet integral.  Differently from  \cite{Lin}  we do not use the heavy machinery of currents and do not derive the  monotonicity formula. Instead, our starting point is the same decay estimate \eqref{decay_result} for the Dirichlet energy used for the study of Taylor cones. This estimate implies that if in a ball $B_r(x_0)$ the perimeter of $E$ is sufficiently small then the total energy  in a smaller ball $B_{\tau r}(x_0)$ 
\[
P(E,B_{\tau r}(x_0)) + \int_{B_{\tau r}(x_0)} |Du_E|^2
\]
is much smaller than the total energy in $B_{r}(x_0)$ (Lemma \ref{decay_energy}). In turn this fact leads to a density lower bound for the perimeter.  

Another  consequence of the estimate  \eqref{decay_result} is that whenever the excess
\[
\E(x_0,r)= \inf_{\nu \in S^{n-1}} \frac{1}{r^{n-1}} \int_{\partial E \cap B_r(x_0)} |\nu_E(x) - \nu|^2\, d \Ha^{n-1} \to 0 \qquad \text{as} \, r \to 0,
 \] 
the Dirichlet integral in $B_r(x_0)$ decays as in \eqref{decay_result}. As in the  Mumford-Shah case this is  one of the two key estimates needed for the regularity proof (see Step 1 of the proof of the Theorem~\ref{mainthm1} at the end of Section 3). Finally the excess decay is proven with a more or less standard argument similar to the one used for the $\Lambda$-minimizers of the perimeter.

\section{Regularity of  elastic minima}

In this section we study the regularity of the elastic minimum associated to a set $E$, i.e., solution of \eqref{linear_eq}.   In the main result of the section, Proposition \ref{decay_est}, we  prove that, if the density of $E$ is close to $0$  or $1$ or the set $E$ is asymptotically close to a hyperplane,  then the elastic energy $\int_{B_{\rho}(x_0)} |Du|^{2}\, dx $ decays faster that  $\rho^{n-1}$. We prove  Proposition \ref{decay_est} with a direct argument and therefore we are  able to  provide explicit bounds for the relevant constants. 

We begin by deriving the Caccioppoli's inequality for solutions of \eqref{linear_eq}. Even though the argument is standard, we  give the proof in order to keep track of the constants. We denote the cube, centred at $x_0$ and with side lenght $2r$, by $Q_r(x_0)$. In the case  $x_0= 0$ we simply write $Q_r$. We  recall the Sobolev-Poincar\'e inequality, i.e., for every function $u \in W^{1,p}(Q_r)$, $1 \leq p < n$, it holds 
\begin{equation} \label{poincare}
||u-u_r||_{L^{p^*}(Q_r)} \leq c(n,p) ||Du||_{L^p(Q_r)}
\end{equation}
where $u_r = \medintdue_{Q_r} u \, dx$ and $p^* = \frac{pn}{n-p}$.

\begin{lemma}
\label{caccioppoli}
Let $u \in W^{1,2}(\Omega)$ be a solution of \eqref{linear_eq}. Then for every cube $Q_{2r}(x_0) \subset \subset \Omega$ it holds
\[
\medint_{Q_r(x_0)} |Du|^2 \, dx \leq  C \, \left(\medint_{Q_{2r}(x_0)} |Du|^{2m} \, dx \right)^{\frac{1}{m}},
\]
where $m = \frac{n}{n+2}$, $C = C_{S,n}^2 \, 2^{n+8}  \frac{\beta}{\alpha}$, and $C_{S,n}$ is the constant in the Sobolev-Poincar\'e inequality \eqref{poincare} with  $p = \frac{2n}{n+2}$.
\end{lemma}

\begin{proof}
Without loss of generality we may assume that $x_0 = 0$. Let  $\zeta \in C_0^{\infty}(Q_{2r})$ be a cut-off function  such that $\zeta \equiv 1$ in $Q_r$ and $|D\zeta| \leq \frac{2}{r}$. We choose a test function $\varphi = (u - u_{2r}) \zeta^2$ in \eqref{linear_eq}, where $u_{2r} =  \medintdue_{Q_{2r}} u \, dx$   and apply Young's inequality to obtain
\[
\int_{Q_{r}} |Du|^2 \, dx  \leq \int_{Q_{2r}} |Du|^2\zeta^2 \, dx \leq  \frac{4 \beta}{\alpha}\int_{Q_{2r}} |u - u_{2r}|^2 | D\zeta|^2 \, dx  \leq \frac{16 \beta}{r^2 \alpha}\int_{Q_{2r}} |u - u_{2r}|^2 \, dx.
\]
We use the Sobolev-Poincar\'e inequality \eqref{poincare} to deduce
\[
\int_{Q_{2r}} |u - u_{2r}|^2 \, dx \leq C_{S,n}^2 \left(\int_{Q_{2r}} |Du|^{2m} \, dx \right)^{\frac{1}{m}}.
\]
The result then follows from the two  inequalities  above.
\end{proof}

We apply Gehring's Lemma to obtain higher integrability for the gradient of $u$. 
\begin{lemma}
\label{gehring}
Let $u \in W^{1,2}(\Omega)$ solve \eqref{linear_eq}. There exists $p >1$ such that for any ball $B_{2r}(x_0) \subset \! \subset \Omega$ it holds
\[
 \medint_{B_r(x_0)}|Du|^{2p} \, dx\leq C \left( \medint_{B_{2r}(x_0)}|Du|^2 \, dx \right)^p.
\]
The constants can be estimated explicitly as
\[
p = \frac{2C_1 - m}{2C_1 -1} \quad \text{for }\,  C_1 = C_{S,n}^2 \, 2^{10}\cdot  80^n  \frac{\beta}{\alpha}  \quad \text{and }\, C =    2^{2n+1} 5^{np} n^{np/2} \omega_n^{p-1},
\]
where $\omega_n$ is the volume of the unit ball and $m=\frac{n}{n+2}$.
\end{lemma}
The above result is well known but it is usually stated without estimates of the constants. In the Appendix  we  will go through  the proof of Lemma \ref{gehring}  from \cite{Giusti} and  evaluate  every  constant explicitly.

In the next lemma we prove a monotonicity formula for the elastic minimum  in the case when $E$ is a half-space.
\begin{lemma}
\label{monotonicity}
Let $E = \{ x \in \R^n \mid \langle x- \bar{x}, e \rangle < 0  \} \cap \Omega$ for some unit vector $e$ and a point $\bar{x}$, and suppose  $u$ is a solution of  \eqref{linear_eq}. Let $ x_0 \in \partial E \cap \Omega$ and $r>0$ be such that $B_r(x_0)  \subset \Omega$. Then 
\[
\rho \mapsto \medint_{B_{\rho}(x_0)}\sigma_E(x) |Du|^2 \, dx 
\]
is increasing in  $(0,r)$.
\end{lemma}

\begin{proof}
Without loss of generality we may assume that $E = \{ x \in \R^n \mid x_n < 0  \} \cap \Omega$ and $x_0= 0$.  Let us fix a radius $r$ such that $B_r \subset \Omega$. From standard elliptic regularity theory we know that $u$ is smooth in the upper and in the  lower part of the ball $B_r$ with respect to the hyperplane $\partial E = \{ x_n= 0\}$. To be more precise, let us denote $\bar{B}_r^+ = \overline{B_r \setminus E}$ and $\bar{B}_r^- = \overline{B_r \cap E}$. Then $u_{\alpha} \in C^{\infty}(\bar{B}_r^+)$ and $u_{\beta} \in C^{\infty}(\bar{B}_r^-)$ and they are harmonic in the interior of $\bar{B}_r^+$ and $\bar{B}_r^-$, where $u_{\alpha}$ and $u_{\beta}$ are the restrictions of $u$ on $\Omega \setminus E$ and $E$.

The goal is to show that the function $\vphi : (0,r) \to \R$
\[
\vphi(\rho) :=  \medint_{\partial B_{\rho}}\sigma_E(x) |Du(x)|^2 \, d \Ha^{n-1}(x) = \medint_{\partial B_1}\sigma_E(\rho y) |Du(\rho y)|^2 \, d \Ha^{n-1}(y) 
\]
is increasing. Notice that $\sigma_E(\rho y) = \sigma_E( y)$ since $E$ is a half-space. Denote $v = |Du|^2$, $v_{\alpha}= |Du_{\alpha}|^2$ and $v_{\beta}= |Du_{\beta}|^2$. Since $v_{\alpha}$ and $v_{\beta}$  are   subharmonic in the interior of  $B_{\rho}^+$ and $B_{\rho}^-$ we deduce by the divergence theorem that 
\[
\begin{split}
\vphi'(\rho) &= \medint_{\partial B_{\rho}}\sigma_{E}(x) \langle Dv(x), \frac{x}{\rho} \rangle \, d \Ha^{n-1}(x) \\
&= \frac{1}{\Ha^{n-1}(\partial B_{\rho})} \left(\alpha  \int_{\partial B_{\rho}^+} \partial_{\nu} v_{\alpha} \, d \Ha^{n-1} + \beta \int_{\partial B_{\rho}^-} \partial_{\nu} v_{\beta} \, d \Ha^{n-1}  +\int_{\partial E  \cap B_{\rho}} \alpha\,  \partial_{x_n} v_{\alpha} - \beta \, \partial_{x_n}v_{\beta}   \, d \Ha^{n-1}  \right)\\
&=  \frac{1}{\Ha^{n-1}(\partial B_{\rho})} \left(\alpha  \int_{ B_{\rho}^+} \Delta v_{\alpha}\, d \Ha^{n-1} + \beta \int_{ B_{\rho}^-} \Delta v_{\beta}\, d \Ha^{n-1}+\int_{\partial E  \cap  B_{\rho}} \alpha\,  \partial_{x_n} v_{\alpha} - \beta \, \partial_{x_n}v_{\beta}   \, d \Ha^{n-1}  \right)\\
&\geq \frac{1}{\Ha^{n-1}(\partial B_{\rho})}\int_{\partial E  \cap B_{\rho}} \alpha\,  \partial_{x_n} v_{\alpha} - \beta \, \partial_{x_n}v_{\beta}   \, d \Ha^{n-1}.
\end{split}
\]
We will show that $\beta \partial_{x_n}v_{\beta} = \alpha\,  \partial_{x_n} v_{\alpha}$ on $\partial E$ from which the claim follows.

Since  $u_{\alpha}= u_{\beta}$ on $\partial E $ we have 
\begin{equation}
\label{the_same}
\partial_{x_i} u_{\alpha} = \partial_{x_i} u_{\beta} \quad \text{and} \quad  \partial_{x_ix_i} u_{\alpha} = \partial_{x_ix_i} u_{\beta} \quad  \text{for }\, i =1, 2, \dots, n-1. 
\end{equation}
The transmission condition \eqref{transition} reads as 
\begin{equation}
\label{trans_plane}
\alpha \, \partial_{x_n} u_{\alpha} = \beta \, \partial_{x_n} u_{\beta}  \quad \text{on }\, \partial E. 
\end{equation}
Differentiating \eqref{trans_plane} with respect to $x_i$, for $ i =1, 2, \dots, n-1$, yields
\[
\alpha \, \partial_{x_i x_n} u_{\alpha} = \beta \, \partial_{x_i x_n} u_{\beta} \quad \text{on }\, \partial E. 
\]
On the other hand, since $u_{\alpha}$ and $u_{\beta}$ are harmonic, we have by \eqref{the_same} that
\[
\partial_{x_n x_n} u_{\alpha} = - \sum_{i=1}^{n-1} \partial_{x_i x_i} u_{\alpha} =   - \sum_{i=1}^{n-1} \partial_{x_i x_i} u_{\beta} = \partial_{x_n x_n} u_{\beta} \quad \text{on }\, \partial E. 
\]
Therefore on $\partial E$ it holds 
\[
 \alpha\,  \partial_{x_n} v_{\alpha} = 2 \sum_{i=1}^{n} \alpha \, \partial_{x_i} u_{\alpha}  \, \partial_{x_i x_n} u_{\alpha} = 2 \sum_{i=1}^{n} \beta \, \partial_{x_i} u_{\beta}  \, \partial_{x_i x_n} u_{\beta}  =  \beta\,  \partial_{x_n} v_{\beta}
\]
which implies $\vphi'(\rho) \geq 0$.

 
\end{proof}

The main result of this section is the following decay estimate for elastic minimum. 

\begin{proposition} 
\label{decay_est}
Let $u \in W^{1,2}(\Omega)$ be a solution of \eqref{linear_eq}. For all $\tau \in (0,1)$ there exists $\eps_0 = \eps_0(\tau)>0$ such that if $B_r(x_0) \subset \! \subset \Omega$ and one of the following conditions hold
\begin{enumerate}
\item[(i)] $\frac{|E \cap B_{r}(x_0)|}{| B_{r}|} <  \eps_0$,  
\item[(ii)] $\frac{|B_{r}(x_0) \setminus E|}{| B_{r}|} <   \eps_0$,  
 \item[(iii)] there exists a half-space $H$ such that   $ \frac{|(E \Delta H ) \cap B_{r}(x_0)|}{| B_{r}|} <  \eps_0$,
\end{enumerate}
then
\[
\int_{B_{\tau r}(x_0)} |Du|^{2}\, dx \leq C_0 \tau^n \int_{B_{r}(x_0)}  |Du|^2 \, dx
\]
for some constant $C_0$ depending only on $\frac{\beta}{\alpha}$ and $n$.
\end{proposition}

\begin{proof}

Without loss of generality we may assume that $\tau < 1/2$. We first treat  the cases (i) and (ii). We fix a ball $B_{r}(x_0)\subset \! \subset \Omega$ and assume without loss of generality that $x_0 = 0$. Choose $v$ to be the harmonic function in $B_{r/2}$ with the boundary condition $v= u$ on $\partial B_{r/2}$. We  choose the test function $\vphi = v-u \in W_0^{1,2}(B_{r/2})$ in the equations 
\[
\int_{B_{r/2}} Dv \cdot D\vphi \, dx   = 0 
\]
and
\begin{equation}\label{decay1} \alpha \int_{B_{r/2} \setminus  E} \ Du \cdot D\vphi\, dx + \beta \int_{B_{r/2} \cap  E}  Du \cdot D\vphi \, dx = 0.
\end{equation} 
 We write the latter equation as
\[
\int_{B_{r/2}}  Du \cdot (Dv - Du) \, dx = - \frac{\beta -  \alpha }{\alpha } \int_{B_ {r/2} \cap E}  Du \cdot (Dv - Du)\, dx.
\]
We substract to this $ \int_{B_{r/2} } Dv \cdot (Dv - Du) \, dx   = 0$  and use H\"older's inequality  to deduce
\[
 \int_{B_{r/2} }  |Dv - Du|^2 \, dx \leq \frac{ (\beta -  \alpha )^2}{\alpha^2}\int_{B_{r/2}\cap E}  |Du|^2 dx.
\] 
By the higher integrability stated in Lemma \ref{gehring} we have
\begin{equation} \label{es_long1}
\begin{split}
 \int_{B_{\tau r}}  |Dv - Du|^2 \, dx &\leq \frac{ (\beta -  \alpha )^2}{\alpha^2}   |E \cap B_{r}|^{1 - 1/p} |B_{r}|^{1/p} \left( \medint_{B_{r}} |Du|^{2p} \right)^{1/p}   \\
&\leq C^{1/p}  \frac{ (\beta -  \alpha )^2}{ \alpha^2} \left( \frac{|E \cap B_{r}|}{| B_{r}|} \right)^{1 - 1/p}  \int_{B_{r}} |Du|^2 \, dx 
\end{split}
\end{equation} 
where $C$ and $p>1$ are from Lemma \ref{gehring}.  Similarly we deduce
\begin{equation} \label{es_long2}
\int_{B_{\tau r }}  |Dv - Du|^2 \, dx  \leq C^{1/p} \frac{ (\beta -  \alpha )^2}{\beta^2} \left( \frac{|E \setminus  B_{r}|}{| B_{r}|} \right)^{1 - 1/p} \int_{B_r} |Du|^2 \, dx.
\end{equation} 
On the other hand, since  $v$ is harmonic, we have
\[
\int_{B_{ \tau r}}  |Dv|^2 \, dx \leq  2^n\tau^n \int_{B_{r/2}}  |Dv|^2 \, dx \leq   2^n \tau^n \int_{B_{r/2}}  |Du|^2 \, dx.
\]
Hence,  we may estimate
\begin{equation} \label{es_long3}
\begin{split}
\int_{B_{\tau r} }  | Du|^2 \, dx &\leq 2 \int_{B_{ \tau r}}   |Dv - Du|^2 \, dx + 2 \int_{B_{\tau r}}   |Dv|^2 \, dx \\
&\leq 2 \int_{B_{\tau r}}   |Dv - Du|^2 \, dx + 2^{n+1} \tau^n  \int_{B_{r}}  |Du|^2 \, dx. 
\end{split}
\end{equation} 
If $\eps_0$ is such that $\eps_0^{1-\frac{1}{p}} = \tau^n$ then \eqref{es_long1}, \eqref{es_long2} and \eqref{es_long3} yield
\[
\int_{B_{\tau r} }  | Du|^2 \, dx \leq C \tau^n \int_{B_{r}}  |Du|^2 \, dx 
\]
for a constant $C$ depending only on $\beta/\alpha$ and $n$.

We are left with the  case (iii). Let $H$ be the half-space from the assumption. We choose  $v$ which minimizes the energy $\int_{B_{r/2}} \sigma_H(x) |Dv|^2 \, dx$ with the boundary condition  $v = u$ on $\partial B_{r/2}$. Hence 
\begin{equation}\label{decay2}
\beta \int_{B_{r/2} \cap H} Dv \cdot D \vphi \, dx + \alpha   \int_{B_{r/2} \setminus H} Dv \cdot D \vphi \, dx = 0 
\end{equation}
for every $\vphi \in W_0^{1,2}(B_{r/2})$. 
Lemma \ref{monotonicity} yields
\[
\int_{B_{\tau r}} |Dv|^2 \, dx \leq  \frac{\beta}{\alpha}  2^n \tau^{n} \int_{B_{{r/2}}}  |Dv|^2 \, dx .
\]
Moreover, from the minimality of $v$ it follows
\[
 \int_{B_{r/2}}  |Dv|^2 \, dx \leq \frac{\beta}{\alpha} \int_{B_{r/2}}  |Du|^2 \, dx.
\]
Hence, we have 
\begin{equation} \label{es_long4}
\int_{B_{\tau r} }  | Du|^2 \, dx \leq 2 \int_{B_{\tau r}}   |Dv - Du|^2 \, dx + 2^{n+1 } \left(\frac{\beta}{\alpha} \right)^2 \tau^{n} \int_{B_{r}}   |Du|^2 \, dx.
\end{equation}

Let us now rewrite the equation \eqref{decay1} satisfied by $u$ as
\[
\begin{split}
\beta \int_{B_{r/2} \cap H} Du \cdot D \vphi \, dx &+ \alpha   \int_{B_{r/2} \setminus H} Du \cdot D \vphi \, dx\\
&=(\beta-\alpha) \int_{B_{r/2} \cap (H \setminus E)} Du \cdot D \vphi \, dx-(\beta-\alpha) \int_{B_{r/2} \cap (E\setminus H)} Du \cdot D \vphi \, dx \,.
\end{split}
\]
Then, subtracting \eqref{decay2} from this equation and choosing $\varphi=u-v$ we get at once
\[
\int_{B_{\tau r}} |Du -Dv|^2 \, dx \leq  \left( \frac{\beta}{\alpha} -1 \right)^2 \int_{B_{r/2} \cap (E \Delta H)}  |Du|^2 \, dx
\]
and from Lemma \ref{gehring} we deduce
\begin{equation} \label{es_long5}
\int_{B_{\tau r} }  |Du -Dv|^2  \, dx \leq C^{1/p} \left( \frac{\beta}{\alpha}-1  \right)^2 \left( \frac{|(E\Delta H) \cap B_{r}|}{| B_{r}|} \right)^{1 - 1/p}  \int_{B_{r}}   | Du|^2 \, dx.
\end{equation}
The conclusion then   follows as in the previous cases.
\end{proof}

\section{Taylor  cones}

In this section we  study  critical configurations $(E,u)$, i.e., they satisfy \eqref{linear_eq} and \eqref{1st_var}. In particular, we are interested in those which are circular cones satisfying \eqref{1st_var} outside the vertex. It was shown in \cite{LHL} and \cite{RC} that there exist circular cones in $\R^3$
\[
E_{\theta_0} = \big{\{ }x \in \R^3 \mid x_3 > \frac{1}{ \tan \theta_0} \sqrt{x_1^2 + x_2^2} \, \big{\}},
\]
for $\theta_0 \in (0, \pi/2)$, which are critical. Indeed, one may find an associated elastic minimum given in spherical coordinates by
\[
u(\rho, \theta) = \sqrt{\rho} \cdot f(\theta),
\] 
where $\rho = \sqrt{x_1^2 + x_2^2 + x_3^2 }$ and $\theta$ is the angle formed by the vector $x \in \R^3$ with the positive $x_3$ semi-axis. Denote by $P_{\frac{1}{2}}$   the Legendre function of the first kind of order $1/2$ which solves  the equation
\[
P''(t) (1- t^2) - 2t P'(t) + \frac{3}{4}P(t)= 0, 	\quad t \in (-1,1).
\]
Then the function  $f$ is given by 
\[
f(\theta) = \left\{
\begin{aligned}
 &P_{\frac{1}{2}}(- \cos \theta_0)  P_{\frac{1}{2}}( \cos \theta), &&  \theta \in [0, \theta_0] \\
 &P_{\frac{1}{2}}( \cos \theta_0)  P_{\frac{1}{2}}(- \cos \theta), &&  \theta \in [\theta_0, \pi].\\
\end{aligned}
\right.
\]
The transmission condition \eqref{transition} then reads as
\[
\beta P_{\frac{1}{2}}(- \cos \theta_0)  P_{\frac{1}{2}}'( \cos \theta_0) + \alpha  P_{\frac{1}{2}}(\cos \theta_0)  P_{\frac{1}{2}}'(- \cos \theta_0) = 0. 
\]
It can be proved that there exists a critical threshold $\lambda_1 \approx 17.59$ such that this equation has no solutions if $\frac{\beta}{\alpha} < \lambda_1$ and it has two solutions in $(0, \pi/2)$ if $\frac{\beta}{\alpha} > \lambda_1 $.

In \cite{SLB} estimates of the angles corresponding to critical cones are given by a different approach. Although the known results  give sharp estimates for the angles which allow  existence of critical spherical cones, they do not give any rigorous answer whether there exists a range of angles where no critical spherical cones appear.

We apply the regularity from Section 2 to prove that only cones with certain angles are possible. This estimate is independent of $\gamma$, which reflects the fact that the perimeter has only regularizing effect. This result rigorously answers to the question connected to Taylor Cones, of why cones of certain angles do not appear. The result also generalizes to convex cones $E$ whose base  is uniformly convex and $C^2$-regular. Since the result is purely local, with no loss of generality we set  $\Omega = \R^n$.

We begin by revisiting the decay estimate proved in Proposition \ref{decay_est} with an explicit choice of the constants. 

\begin{proposition} 
\label{decay_est_2}
Let $u \in W_{loc}^{1,2}(\R^n)$ be a solution of \eqref{linear_eq}. There exist $\delta_1, \sigma >0$ and $\vartheta \in (0,1)$, depending only on the dimension $n$ and the ratio $\frac{\beta}{\alpha}$, such that if one of the following conditions hold
\begin{enumerate}
\item[(i)] $|E \cap B_{1}|< \delta_1|B_1|$,  
\item[(ii)] $|B_1 \setminus E| <  \delta_1|B_1|$,  
 \item[(iii)] there exists a half-space $H$ such that   $ |(E \Delta H ) \cap B_1| < \delta_1|B_1|$,
\end{enumerate}
then we have
\[
\int_{B_{\vartheta}} |Du|^{2}\, dx \leq \vartheta^{n-1 +  \sigma} \int_{B_{1}}  |Du|^2 \, dx.
\]
\end{proposition} 

\begin{proof}
We first deal the case (i). We recall that from \eqref{es_long1} and \eqref{es_long3} it follows that if $\tau \in (0, \frac{1}{2})$ then we have
\[
\int_{B_{\tau}} |Du|^{2}\, dx \leq 2\left( C^{1/p} \frac{(\beta- \alpha)^2}{\alpha^2 } \delta_1^{1-1/p}+ 2^n\tau^n \right)\int_{B_{1}} |Du|^{2}\, dx
\]
where $C$ and $p$ are the constants from Lemma \ref{gehring}. Note that this inequality is trivially satisfied if $\tau \in (\frac{1}{2},1)$. Let us denote by $\chi$ the largest number such that the equation 
\begin{equation}
\label{explicit1.1}
2\chi +  2^{n+1} \vartheta^{n} = \vartheta^{n-1}
\end{equation}
has  a solution for some $\vartheta>0$.  We may easily solve $\vartheta$
\[
\vartheta= \frac{n-1}{ 2^{n+1}n}.
\]
Then if 
\[
C^{1/p} \frac{(\beta- \alpha)^2}{\alpha^2 }  \delta_1^{1-1/p}  < \chi  
\]
we get 
\begin{equation}
\label{theta}
\int_{B_{\vartheta  } }  | Du|^2 \, dx \leq \vartheta^{n-1 + 2 \sigma} \int_{B_{1}}  |Du|^2 \, dx,
\end{equation}
for some $\sigma>0$. Similarly in the case (ii) we have \eqref{theta} provided that we choose $\chi$ and $\vartheta$ as in \eqref{explicit1.1} and $\delta_1$ such that 
\[
C^{1/p} \frac{(\beta- \alpha)^2}{\beta^2 }  \delta_1^{1-1/p}  < \chi.  
\]
In the case (iii) we choose  $\chi$ the largest number such that  the equation 
\[
2\chi +   2^{n+1} \left( \frac{\beta}{\alpha} \right)^2 \vartheta^n = \vartheta^{n-1}
\]
has a solution $\vartheta>0$. We may again solve $\vartheta$ 
\begin{equation}
\label{theta2}
\vartheta= \frac{n-1}{2^{n+1 }n} \left( \frac{\alpha}{\beta}\right)^2.
\end{equation}
Arguing as before, the estimate \eqref{theta} follows from \eqref{es_long4} and \eqref{es_long5} if we choose $\delta_1$ such that 
\begin{equation}
\label{explicit4}
 C^{1/p}  \left( \frac{\beta}{\alpha} \right)^2   \delta_1^{1 - 1/p} < \chi.
\end{equation}
By comparing the above three cases the claim obviously follows by choosing 
\begin{equation}
\label{explicit5}
\delta_1^{1-1/p} < \left(\frac{\alpha}{\beta}\right)^2 C^{-1/p} \chi
\end{equation}
where $\chi = \frac{1}{2}\vartheta^{n-1}- 2^n\left(\frac{\beta}{\alpha}\right)^2 \vartheta^n$ and $\vartheta$ is from \eqref{theta2}.
\end{proof}

We recall the following result (see the proof of \cite[Theorem 3.1]{EF})
\begin{proposition}
\label{espfus}
Let $u \in W^{1,2}(\Omega)$ be a solution of \eqref{linear_eq}. Let  $x_0 \in \Omega$ and $r>0$ such that $B_{r}(x_0) \subset \! \subset \Omega$. There exist $\lambda_0 >1, c>0$ and $\sigma  >0$, depending only on  $n$,  such that  if  
\[
\frac{\beta}{ \alpha} <   \lambda_0 
\]
then  for every $\rho < r$
\[
\int_{B_{\rho}(x_0)} |Du|^{2}\, dx \leq c \left(\frac{\rho}{r}\right)^{n-1 +  \sigma} \int_{B_{r}(x_0)}  |Du|^2 \, dx.
\]
\end{proposition}

We remark that we may estimate the number $\lambda_0$  (see  \cite[(20)-(21)]{EF}) by 
\begin{equation} \label{estimating_gamma}
\lambda_0 = \frac{n^n + n(n-1)^{n-1} - (n-1)^{n}}{n^n - n(n-1)^{n-1} + (n-1)^{n}}.
\end{equation}

\begin{proof}[\textbf{Proof of Theorem \ref{mainthm2}}]
 Let us recall the Euler-Lagrange equation for the critical set $E_{\theta}$ 
\[
\gamma H_{\partial E_{\theta}} + \beta |D u_{\beta}|^2  - \alpha |D u_{\alpha}|^2 = \lambda \quad \text{ on }\, \partial E_{\theta} \setminus \{0\}.
\]
This can be rewritten as 
\begin{equation}
\label{euler_re}
\gamma H_{\partial E_{\theta}} + \beta |\partial_{\nu}u_{\beta}|^2  - \alpha |\partial_{\nu} u_{\alpha}|^2 + (\beta - \alpha)|D_{\tau}u|^2= \lambda \quad \text{ on }\, \partial E_{\theta} \setminus \{0\},
\end{equation}
where $D_{\tau}u$ is the tangential gradient of $u$ on $\partial E_{\theta} \setminus \{0\}$. From the transmission condition \eqref{transition}  we deduce
\begin{equation}
\label{simple}
\beta |\partial_{\nu} u_{\beta}|^2  < \alpha |\partial_{\nu}  u_{\alpha}|^2 \quad \text{ on }\, \partial E_{\theta}\setminus \{0\}.
\end{equation}
Since $H_{\partial E_{\theta}}(x) = (n-1)\cos \theta \cdot |x|^{-1}$  for $x \in \partial E \setminus \{0\}$, we obtain from the Euler-Lagrange equation \eqref{euler_re}, from \eqref{simple}, and from the transmission condition \eqref{transition} that
\begin{equation}
\label{curv_bound}
|\partial_{\nu}u_{\beta}(x)| \geq \frac{c}{\sqrt{|x|}} \qquad \text{on }\, \partial E \setminus \{0\},
\end{equation}
for some constant $c>0$. In particular, since the set $ \partial E \setminus \{0\}$ is connected, this implies that $\partial_{\nu}u_{\beta}$ does not change sign on $\partial E \setminus \{0\}$, and we may thus  assume it to be  positive.

Let us fix $\rho >0$ and choose a cut-off function $\zeta \in C_0^{\infty}(B_{\rho})$ such that $\zeta \equiv  1$ in $B_{\rho/2}$ and $|D \zeta| \leq 4/\rho$. Since $u_{\beta}$ is harmonic in $E$, we obtain  from  \eqref{curv_bound} and by  integrating  by parts that  
\[
\int_{E \cap B_{\rho}} \langle D u_{\beta},  D \zeta \rangle \, dx = \int_{\partial E \cap B_{\rho}}  \partial_{\nu}u_{\beta}\, \zeta \, d \Ha^{n-1} \geq c \int_{\partial E \cap B_{\rho/2}}  |x|^{-1/2}\, \Ha^{n-1} \geq \tilde{c} \rho^{n -3/2}.
\]
On the other hand H\"older's inequality implies
\[
\begin{split}
\int_{E \cap B_{\rho}} \langle D u_{\beta},  D \zeta \rangle \, dx &\leq\left( \int_{E \cap B_{\rho}} |D \zeta|^2 \, dx \right)^{1/2} \left( \int_{E \cap B_{\rho}} |D u_{\beta}|^2 \, dx \right)^{1/2}  \\
&\leq C \rho^{n/2 -  1} \left( \int_{E \cap B_{\rho}} |D u_{\beta}|^2 \, dx \right)^{1/2}.
\end{split}
\]
Therefore 
\[
 \int_{E \cap B_{\rho}} |D u_{\beta}|^2 \, dx  \geq c_0 \rho^{n-1}
\]
for some constant $c_0>0$. The claim now follows  by a standard iteration argument  from Proposition~\ref{decay_est_2} (i) and  (iii) and from Proposition \ref{espfus}.
 \end{proof}

\begin{remark}
The constant $\delta_0$ can be explicitly estimated in terms of the constant $\delta_1$ from Proposition~\ref{decay_est_2}, since the spherical sector has the volume
\[
|E_{\theta} \cap B_1|= \omega_{n-1}\left( \int_0^{\theta} \sin^{n} t \, dt + \frac{\sin^{n-1} \theta \cos \theta}{n}\right),
\]
where $\omega_{n-1}$ is the volume of the $(n-1)$-dimensional unit ball. The formula for $\delta_1$ is given by \eqref{explicit5}.  The constant $\lambda_0$ is estimated in \eqref{estimating_gamma}. 
\end{remark}

Note that in dimension 2 the Euler-Lagrange equation \eqref{1st_var} reduces to
$$
\beta |D u_{\beta}|^2  - \alpha |D u_{\alpha}|^2 = \lambda \quad \text{ on }\, \partial E_{\theta} \setminus \{0\}
$$
and it is not clear to us if this weaker information is enough to establish Theorem~\ref{mainthm2} also in this case. However, as we already mentioned earlier, for $n\geq3$, the proof of  that theorem can be easily generalized to more general cones.
\begin{remark}
 Let $K \subset \R^{n-1}$ be an open,  uniformly convex and $C^2$-regular set such that $0 \in K$. Theorem~\ref{mainthm2} can be generalized to  conical sets of the form
\[
E = \{ \lambda(x', 1) \in \R^n \mid x' \in K, \,\, \lambda \geq 0\}.
\]
\end{remark}

\section{Regularity of minimizers}

Throughout this section the dimension $n$ and the constants $\alpha, \beta$ and $\gamma$ will remain fixed. Thus we denote by $C$ a generic constant depending on these quantities and whose value may change from line to line. On the other hand special constants will be numbered and their dependence on other quantities will be explicitly mentioned.  Before proceeding in the regularity proof we recall the following result which was proved in \cite[Theorem 1]{EF}.
\begin{theorem} \label{fusco_esposito}
Let $(E,u)$ be a minimizer of  problem  $(P_c)$. There exists a constant $\Lambda$ such that $(E,u)$ is also a minimizer of the penalized functional
\[
\F_{\Lambda}(F, v) = \gamma P(F, \Omega) +  \int_{\Omega} \sigma_F (x)|D v|^2 \, dx + \Lambda \big| |F| - |E|\big|
\]
among all $(F,v)$ such that $v = u$ on $\partial \Omega$.
\end{theorem}

Motivated by the previous theorem we give the following definition. To this aim we denote for a set $E$ with  finite perimeter in  $\Omega$  and a function $u \in W^{1,2}(\Omega)$ by
\[
\F(E, u; U) = \gamma P(E,U) + \int_U \sigma_E |D u|^2\, dx
\]
the energy of the pair $(E,u)$  in an open set $U \subset \! \subset \Omega$.  We say that  a pair $(E,u)$   is a \emph{$\Lambda$-minimizer of $\F$ in $\Omega$} if for every $B_r(x_0)  \subset  \Omega$ and every pair $(F,v)$, where $F$ is a set  of finite perimeter   with $F \Delta E \subset \! \subset B_r(x_0)$ and $v-u \in W_0^{1,2}(B_r(x_0))$, we have
\[
\F(E, u; B_r(x_0)) \leq \F(F, v; B_r(x_0)) + \Lambda |F \Delta E|. 
\]
Note that any minimizing pair $(E,u_E)$ of the constrained problem $(P_c)$ is a $\Lambda$-minimizer of $\F$ for some $\Lambda$. Since $\Lambda$ will be fixed from now on  the dependence of the constants on $\Lambda$ will not be highlighted.

\begin{lemma}
\label{decay_energy}
Let $(E,u)$ be a $\Lambda$-minimizer of $\F$  in $\Omega$. For every $\tau \in (0,1)$ there exists $\eps_1 = \eps_1(\tau) >0$ such that if  $B_r(x_0) \subset \Omega$ and $P(E; B_r(x_0)) < \eps_1 r^{n-1}$ then
\[
\F(E,u; B_{\tau r}(x_0)) \leq C_1 \tau^n \left( \F(E,u; B_{ r}(x_0)) +   r^n \right)
\] 
for some constant $C_1$ independent of $\tau$ and $r$.
\end{lemma}

\begin{proof}
If we replace $E$ by $\frac{E-x_0}{r}$, $u$ by $y \to r^{-1/2}u(x_0 + ry)$ and $\Lambda$ by $\Lambda r$ we may assume that $r=1$, $x_0 = 0$ and $(E,u)$ is a $\Lambda r$-minimizer in $\frac{\Omega-x_0}{r}$. After this rescaling we are left to prove that for a given $\tau \in (0,1)$ there exists $\eps_1 = \eps_1(\tau)>0$ such that if $P(E, B_1) < \eps_1$ then
\[
\F(E,u; B_{\tau }) \leq C_1 \tau^n \left( \F(E,u; B_{1}) +   \Lambda r \right).
\]
Without loss of generality we may assume that $\tau < 1/2$.

If the perimeter of $E$ in $B_1$ is small then, by the relative isoperimetric inequality, either $|B_1 \cap E|$ or $|B_1 \setminus E|$ is small. Assume the latter is true. Then we have 
\[
|B_1 \setminus E| \leq c(n) P(E, B_1)^{\frac{n}{n-1}}.
\]
By Fubini's theorem and choosing as a representative of $E$ the set of points of density one, we have
\[
|B_1 \setminus E| \geq \int_\tau^{2\tau} \Ha^{n-1}(\partial B_\rho \setminus E)\, d \rho.
\]
Therefore we may choose $\rho \in (\tau, 2 \tau)$ such that
\begin{equation}
\label{use_isoperimetric}
 \Ha^{n-1}(\partial B_\rho \setminus E) \leq \frac{c(n)}{\tau} P(E, B_1)^{\frac{n}{n-1}} \leq \frac{c(n) \eps_1^{\frac{1}{n-1}}}{\tau} P(E, B_1).
\end{equation}

Set $F = E \cup B_\rho$ and observe that by the choice of the representative of $E$ it holds
\[
P(F;B_1) \leq P(E,B_1\setminus \bar{B}_\rho) + \Ha^{n-1}(\partial B_\rho \setminus E).
\]
Choosing $(F,u)$ as a competing pair and using the $\Lambda r$-minimality of $(E,u)$ we get
\[
\begin{split}
\gamma P(E,B_1) &+ \int_{B_1} \sigma_E |Du|^2 \, dx \leq \gamma P(F;B_1) + \int_{B_1} \sigma_F |Du|^2 \, dx + \Lambda r |F \setminus E| \\ 
&\leq \gamma ( P(E,B_1 \setminus \bar{B}_\rho) + \Ha^{n-1}(\partial B_\rho \setminus E) ) + \int_{B_1} \sigma_F |Du|^2 \, dx + \Lambda r  |B_\rho|.
\end{split}
\] 
Hence, choosing $\eps_1$ such that  $c(n) \eps_1^{\frac{1}{n-1}} \leq \tau^{n+1}$ and recalling that  $\rho \in (\tau, 2 \tau)$ we get from  \eqref{use_isoperimetric} that
\[
\gamma P(E,B_\tau) + \int_{B_\tau} \sigma_E |Du|^2 \, dx \leq \gamma \tau^n P(E,B_1) +  \beta \int_{B_{2 \tau}} |Du|^2 \, dx+ \Lambda r |B_{2 \tau}|.
\]
If we choose $\eps_1$ such that $c(n)\eps_1^{\frac{n}{n-1}} \leq \eps_0(2 \tau) |B_1|$, where $\eps_0$ is from Proposition \ref{decay_est}, we get
\[
\int_{B_{2 \tau}} |Du|^2 \, dx \leq 2^n C_0 \tau^n \int_{B_1} |Du|^2 \, dx
\]
and the result follows.

\end{proof}

The next result is contained in \cite[Theorem 2]{Lin}, where it is proven for local minimizers of $\F$. However, in the case of $\Lambda$-minimizers the same proof applies without changes once one observes that by comparing $(E,u)$ with $(F,u)$, where $F = E \setminus B_r(x_0)$ and $B_r(x_0) \subset \! \subset  \Omega$, one gets
\[
\gamma P(E,B_r(x_0)) + (\beta-\alpha) \int_{E \cap B_r(x_0) } |Du|^2\, dx  \leq \gamma \Ha^{n-1}(\partial B_r \cap E)+ \Lambda |B_r|  \leq C r^{n-1}. 
\]
\begin{theorem}
\label{energy_upper}
Let $(E,u)$ be a $\Lambda$-minimizer of $\F$ in $\Omega$. For every open set $U \subset \! \subset  \Omega$  there exists a constant $C_2$, depending on $U$ and $||Du||_{L^2(\Omega)}$, such that for every  $B_r(x_0) \subset U$ it holds
\[
\F(E,u; B_{ r}(x_0)) \leq C_2 r^{n-1}.
\] 
\end{theorem}

By combining the energy decay Lemma \ref{decay_energy} with the energy upper bound given by Theorem \ref{energy_upper} we obtain the following density lower bound by a standard iteration argument.  From now on the topological boundary $\partial E$ must be understood by considering the correct representative of the set (see \cite[Proposition~12.19]{M}). In particular, for such  representative of $E$ it holds $\overline{\partial^* E} = \partial E$.

\begin{proposition}
\label{density_estimate}
Let $(E,u)$ be a $\Lambda$-minimizer of $\F$  and   $U \subset \! \subset  \Omega$.  There exists a constant $c>0$, depending on $U$ and $||Du||_{L^2(\Omega)}$,  such that for every  $B_r(x_0) \subset U$  with $x_0 \in \partial E$ we have
\[
P(E,B_r(x_0)) \geq cr^{n-1}.
\] 
Moreover $\Ha^{n-1}( (\partial E \setminus  \partial^* E) \cap \Omega) = 0$.
\end{proposition}

\begin{proof}
The proof is as in \cite[Theorem 7.21]{AFP}  and in fact,  even simpler. Fix $x_0 \in \partial^* E$. Without loss of generality we may assume that $x_0 = 0$.  Fix $\tau \in (0,1)$ such that $2C_1 \tau^{\frac{1}{2}} < 1$ and $\sigma \in (0,1)$ such that $2C_1C_2 \sigma < \eps_1(\tau)\gamma$ and $r_0$ such that $r_0 < \min \{ \eps_1(\tau)\gamma, C_2\}$, where $C_1, \eps_1$ and $C_2$ are the constants from Lemma \ref{decay_energy} and Theorem \ref{energy_upper}. Assume by contradiction that for some  $B_r \subset U$ with $r < r_0$  we have $P(E,B_r) < \eps_1(\sigma)r^{n-1}$. Then  by induction we deduce that 
\begin{equation} \label{induction}
\F(E,u;B_{\sigma \tau^h r}) \leq \eps_1(\tau) \gamma \tau^{\frac{h}{2}}(\sigma \tau^h r)^{n-1}
\end{equation}
for all $h = 0,1,2 , \dots$. Indeed, when $h = 0$, using  Lemma \ref{decay_energy} and Theorem \ref{energy_upper} and recalling our choice of $\sigma$ and $r_0$ we obtain
\[
\F(E,u;B_{\sigma  r}) \leq C_1\sigma^n(\F(E,u;B_{ r}) + r^n) \leq 2C_1C_2 \sigma( \sigma r)^{n-1} < \eps_1(\tau) \gamma ( \sigma r)^{n-1}.
\]
If \eqref{induction} holds for $h$ to deduce that it also holds for $h+1$ it is enough to apply Lemma \ref{decay_energy} and to recall that $2C_1 \tau^{\frac{1}{2}} < 1$ and that $r <r_0 <\eps_1(\tau) \gamma$. In particular, this implies   
\[
\lim_{r \to 0} \frac{P(E,B_r) }{r^{n-1}} = 0,
\]
which is a contradiction since $x_0 \in \partial^* E$. This proves that for every $x_0 \in \partial^* E$ and $r <r_0$ it holds $P(E,B_r(x_0)) \geq \eps_1(\sigma)r^{n-1}$. The claim follows from the fact that $\overline{\partial^* E} = \partial E$. The fact that $\Ha^{n-1}( (\partial E \setminus  \partial^* E) \cap \Omega) = 0$ follows from the density lower bound and  \cite[(2.42)]{AFP}.
\end{proof}

Let $(E,u)$  be a $\Lambda$-minimizer of $\F$.  We introduce \emph{the excess of $E$} at the point $x \in \partial E$ at the scale $r>0$  in direction $\nu \in  S^{n-1}$
\[
\mathcal{E}(x,r, \nu) := \frac{1}{2r^{n-1}} \int_{\partial E \cap B_r(x)} |\nu_E(y) - \nu|^2 \, d \Ha^{n-1}(y)
\]
and 
\[
\E(x,r) := \min_{\nu \in S^{n-1}} \E(x,r, \nu) 
\] 
and \emph{the rescaled Dirichlet integral of $u$}
\[
\D(x,r) := \frac{1}{r^{n-1}} \int_{B_r(x)} |Du|^2\, dy.
\]

The density lower bound implies the following important properties of $\Lambda$-minimizers of $\F$. The first one is the so called  height bound lemma. 
\begin{lemma}[Height bound]
\label{height_bound}
Let $(E,u)$ be a $\Lambda$-minimizer of $\F$  in $B_r(x_0)$. There exist $C$ and $\eps>0$, depending  on $||Du||_{L^2(B_r(x_0))}$,  such that if $x_0 \in \partial E$ and 
\[
\E(x_0,r,\nu) < \eps
\]
for some $\nu  \in S^{n-1}$ then 
\[
\sup_{y \in \partial E \cap B_{r/2}(x_0)} \frac{|\nu \cdot (y-x_0)|}{r} \leq C \E(x_0,r,\nu)^{\frac{1}{2(n-1)}}.
\]
\end{lemma}

\begin{proof}
The proof of this result can be obtained arguing exactly  as in the case of  $\Lambda$-minimizers of the perimeter  \cite[Theorem~22.8]{M}. Indeed, it  is based only on the relative isoperimetric inequality, the density lower bound, i.e., Proposition  \ref{density_estimate} and the compactness result below. 
\end{proof}

The next result is proved as in the case of the $\Lambda$-minimizers of the perimeter with the obvious changes due to the presence of the Dirichlet integral. 
\begin{lemma}[Compactness]
\label{compactness}
Let $(E_h,u_h)$ be a sequence of  $\Lambda_h$-minimizers of $\F$ in $\Omega$ such that $\Lambda_h \to \Lambda \in [0,\infty)$ and $\sup_h \F(E_h,u_h; \Omega)< \infty$. Then there exist a subsequence, not relabelled, and a $\Lambda$-minimizer $(E,u)$ of $\F$ such that  for every open set  $U \subset \! \subset \Omega$, $E_h \to E$ in $L^1(U)$, $P(E_h,U) \to P(E,U)$, $u_h \to u$  in $W^{1,2}(U)$ and moreover
\begin{enumerate}
\item[(i)] if $x_h \in \partial E_h \cap U$ and  $x_h \to x \in U$, then $x \in \partial E \cap U$.
\item[(ii)]  if $x \in \partial E \cap U$ there exists  $x_h \in \partial E_h \cap U$ such that  $x_h \to x$.
\end{enumerate}
If in addition, $Du_{h} \rightharpoonup 0$ weakly in $L_{loc}^2(\Omega, \R^n)$ and $\Lambda_h \to 0$, then $E$ is a local minimizer of the perimeter, i.e., for every $F$ such that $F \Delta E \subset \! \subset B_r(x_0) \subset \Omega$ it holds
\[
P(E, B_r(x_0) ) \leq P(F, B_r(x_0) ).
\] 
\end{lemma}

\begin{proof}
We start by proving the $\Lambda$-minimality of $(E,u)$. Let us fix $B_r(x_0) \subset\! \subset \Omega$ and assume that $x_0 =0$. Without loss of generality we may  assume that  $E_h \to E$ in $L^1(B_r)$, $u_h \rightharpoonup u$ weakly in $W^{1,2}(B_r)$ and strongly in $L^2(B_r)$.  Let $(F,v)$ be a pair such that $F \Delta E \subset \! \subset B_r$ and $\text{supp}(u-v) \subset \! \subset B_r$.  By Fubini's theorem and passing to  a subsequence if necessary we may choose $\rho < r$ such that $F \Delta E \subset \! \subset  B_{\rho}$,  $\text{supp}(u-v) \subset \! \subset B_{\rho}$,
\[
\Ha^{n-1}(\partial^* F \cap  \partial B_{\rho} ) = \Ha^{n-1}(\partial^* E_h \cap  \partial B_{\rho} ) = 0 \quad \text{and} \quad \lim_{h \to 0} \Ha^{n-1}(\partial B_{\rho} \cap E \Delta E_h) = 0
\]
where it is understood that $E$ and $E_h$ stand for the set of points of density one of $E$ and $E_h$, respectively. Fix a cut-off function $\zeta \in C_0^1(B_r)$ such that $\zeta \equiv 1$ in $B_{\rho}$, and choose $F_h = (F\cap B_{\rho}) \cup (E_h\setminus B_{\rho})$ and $v_h = \zeta v + (1- \zeta) u_h$. Then by the $\Lambda_h$-minimality of $(E_h,u_h)$, the choice of $\rho$, the strong convergence of $u_h\to u$  in $L^2$ and by convexity, we have
\begin{equation} \label{compare}
\begin{split}
&\gamma P(E_h,B_r) + \int_{B_r} \sigma_{E_h} |Du_h|^2 \, dx \leq \gamma P(F_h,B_r) + \int_{B_r} \sigma_{F_h} |Dv_h|^2 \, dx + \Lambda_h|F_h \Delta E_h| \\
&\leq \gamma [P(F, B_{\rho}) + P(E_h, B_r \setminus \bar{B}_{\rho})]+ \int_{B_r} \sigma_{F_h} \zeta  |Dv|^2 \, dx + \int_{B_r} \sigma_{E_h}(1-\zeta)|Du_h|^2 \, dx  + \eps_h+  \Lambda_h|F_h \Delta E_h|
\end{split}
\end{equation}
for some $\eps_h \to 0$. Thus by a simple lower semicontinuity argument we have
\[
\gamma P(E,B_\rho) + \int_{B_r} \sigma_{E} \zeta|Du|^2 \, dx\leq \gamma P(F, B_{\rho}) +  \int_{B_r} \sigma_{F} \zeta |Dv|^2 \, dx + \Lambda |F \Delta E|.
\]
 Letting $\zeta \downarrow \chi_{B_\rho}$  in the previous inequality we conclude that 
\begin{equation} \label{easy_lowersemi}
\gamma P(E,B_\rho) + \int_{B_\rho} \sigma_{E} |Du|^2 \, dx \leq  \gamma P(F, B_{\rho}) +  \int_{B_\rho} \sigma_{F} |Dv|^2 \, dx  + \Lambda |F \Delta E|
\end{equation}
thus proving the $\Lambda$-minimality of $(E,u)$. Similarly, choosing $F = E$  and $v=u$ in \eqref{compare} and arguing as before we get
\[
\limsup_{h \to \infty}\left( \gamma P(E_h,B_\rho) + \int_{B_r} \sigma_{E_h}\zeta |Du_h|^2 \, dx\right)  \leq  \gamma P(E, B_{\rho}) +  \int_{B_r} \sigma_{F} \zeta |Du|^2 \, dx. 
\] 
 Letting $\zeta \downarrow \chi_{B_\rho}$ we conclude that 
\[
\lim_{h \to \infty} P(E_h,B_\rho) =  P(E, B_{\rho})  \qquad \text{and} \qquad \lim_{h \to \infty} \int_{B_\rho} \sigma_{E_h} |Du_h|^2 \, dx=  \int_{B_\rho} \sigma_{E} |Du|^2 \, dx. 
\]
A standard argument then implies that $P(E_h,U) \to P(E,U)$, $u_h \to u$  in $W^{1,2}(U)$ for every open set $U \subset \!\subset \Omega$.

Finally we note that if $Du = 0$ we can choose $v =u$ in \eqref{easy_lowersemi} thus proving that $E$ is a $\Lambda$-minimizer of the perimeter.  The claim (i) and (ii) follow exactly as in \cite[Theorem 21.14]{M}.

\end{proof}

The next  consequence of the density lower bound is the  Lipschitz approximation. Its  proof is based only on the height bound estimate and can be obtained by following word by word the proof given in \cite[Theorem 23.7]{M}. To that aim we use the notation $x = (x',x_n) \in \R^{n-1}\times \R$  for a generic point in $\R^n$ and with a slight abuse of notation we still denote by $Df$ the gradient of a function $f: \R^{n-1} \to \R$. 

\begin{proposition}[Lipschitz approximation]
\label{lipschitz_approx}
Let $(E,u)$ be a $\Lambda$-minimizer of $\F$  in $B_{r}(x_0)$.  There exist $C_3$ and $\eps_3>0$, depending on $||Du||_{L^2(B_r(x_0))}$,  such that if $x_0 \in \partial E$  and 
\[
\E(x_0,r, e_n) < \eps_3
\]
then there exists a Lipschitz function $f : \R^{n-1} \to \R$ such that 
\[
\sup_{x' \in \R^{n-1}} \frac{|f(x')|}{r} \leq C_3 \E(x_0,r, e_n)^{\frac{1}{2(n-1)}}, \qquad ||Df ||_{L^\infty} \leq 1
\]
and
\[
\frac{1}{r^{n-1}} \Ha^{n-1}(\partial E \Delta \Gamma_f \cap B_{r/2}(x_0)) \leq C_3 \E(x_0,r, e_n)
\]
where $\Gamma_f $ is the graph of $f$. Moreover
\[
\frac{1}{r^{n-1}}\int_{B_{r/2}^{n-1}(x_0')} |Df|^2 \, dx' \leq C_3 \E(x_0,r, e_n).
\]
\end{proposition}

Finally  we state the following reverse Poincar\'e inequality which plays the role of the classical Caccioppoli inequality in the elliptic regularity theory. The proof can be obtained exactly as in the case of $\Lambda$-minimizers of the perimeter by constructing a suitable squashing-deformation $F$ of the set $E$ and comparing the energies at $(E,u)$ and $(F,u)$, see \cite[Theorem 24.1]{M}.

\begin{theorem}
\label{tilt_lemma}
Let $(E,u)$ be a $\Lambda$-minimizer of $\F$ in $B_r(x_0)$. There exist two constants $C_4$  and $\eps_4>0$ such that  if $x_0 \in \partial E$  and $\E(x_0,r,\nu) < \eps_4$ then 
\[
\E(x_0,r/2,\nu) \leq C_4\left( \frac{1}{r^{n+1}} \int_{\partial E\cap B_r(x_0)} |\nu \cdot (x-x_0) -  c|^2 \, d \Ha^{n-1} + \D(x_0,r) + r \right)
\]
for  every $c \in \R$. 
\end{theorem}

While in our case the proofs for the  height bound, the compactness, the  reverse Poincar\'e and the Lipschitz approximation are exactly as in the case of $\Lambda$-minimizers of the perimeter, the next step in the regularity proof, i.e., the excess decay, is different. This is  due to the interplay between the excess and the rescaled Dirichlet integral. We follow an argument similar to the one used in proving the flatness decay for the minimizers of the  Mumford-Shah functional, see \cite[Theorem 8.15]{AFP}. To this aim we  first prove the following weaker form of the Euler-Lagrange equation.

\begin{proposition}
\label{euler2}
Let $(E,u)$ be a $\Lambda$-minimizer of $\F$ in $B_r(x_0)$. For every vector field $X \in C_0^1(B_r(x_0); \R^n)$ it holds
\begin{equation}\label{euler21}
\gamma \int_{\partial E} \diver_{\tau}X \, d \Ha^{n-1} + \int_\Omega \sigma_E (|Du|^2\diver X- 2 \langle DX  Du, Du \rangle ) \, dx \leq \Lambda \int_{\partial E} |X| \, d \Ha^{n-1},
\end{equation}
where $\diver_{\tau}$ denotes the tangential divergence on $\partial E$.
\end{proposition}

\begin{proof}
The argument is similar to the one  in \cite[Theorem 7.35]{AFP} and therefore we only give the sketch of the proof.  For a given $X \in C_0^1(B_r(x_0); \R^n)$ we set for every small $t >0$,  $\Phi_t(x)= x - tX(x)$,  $E_t = \Phi_t(E)$ and $u_t(y)= u(\Phi_t^{-1}(y))$. From the $\Lambda$-minimality it follows
\[
\gamma [P(E_t,B_r(x_0))- P(E, B_r(x_0)) ]+ \int_{B_r(x_0)} (\sigma_{E_t} |Du_t|^2  -\sigma_{E} |Du|^2) \, dx + \Lambda |E_t \Delta E| \geq 0. 
\]
The conclusion then follows with the same  standard calculations used to derive the first variation of the perimeter and the Dirichlet integral (see \cite[Theorem 7.35]{AFP})  and observing that 
\[
\lim_{t \to 0}\frac{|E_t \Delta E|}{t} \leq  \int_{\partial E} |X\cdot \nu_E| \, d \Ha^{n-1},
\]
see for instance \cite[Theorem 3.2]{JP}.
\end{proof}

We are ready to prove the excess improvement.

\begin{proposition}
\label{flatness_impro}
Let $(E,u)$ be a $\Lambda$-minimizer of $\F$  in $B_r(x_0)$.  For every $\tau \in (0,1/2)$ and $M$ there exists $\eps_5=\eps_5(\tau, M) \in (0,1)$ such that if $x_0 \in \partial E$ and 
\[
\E(x_0,r) \leq \eps_5 \quad \text{and} \quad  \D(x_0,r) + r \leq M \E(x_0,r)
\]
then 
\[
\E(x_0,\tau r) \leq C_5 (\tau^2 \E(x_0,r) + \D(x_0,2\tau r) + \tau r)
\]
for some constant $C_5$ dependening on $\F(E,u; B_r(x_0))$.
\end{proposition}

\begin{proof}
Without loss of generality we may assume that $\tau < 1/8$.  We argue by contradiction. After performing the same translation and rescaling used in the proof of Lemma~\ref{decay_energy}, we may assume that there exist an infinitesimal sequence $\eps_h$, a sequence $r_h\in\R$ and a sequence $(E_h,u_h)$ of $\Lambda r_h$-minimizers of $\F$ in  $B_1$, with equibounded energies such that, denoting by $\E_h$ the excess of $E_h$ and by $\D_h$ the rescaled Dirichlet integral of $u_{h}$, we have
\begin{equation}\label{flatness2}
\E_h(0,1):=\eps_h, \quad \quad \D_h(0,1) + r_h \leq M \eps_h,
\end{equation}
and
\begin{equation}\label{flatness3}
\E_h(0,\tau) > C_5 (\tau^2 \E_h(0,1) + \D_h(0,2\tau) + \tau r_h)
\end{equation}
Moreover, up to rotating each $E_h$ if necessary we may also assume, that for all $h$
$$
\E_h(0,1)=\E_h(0,1,e_n)=\frac12\int_{\partial E_h \cap B_1}\bigl|\nu_{E_h}(x)-e_n\bigr|^2\,d\Ha^{n-1}.
$$
\par\noindent {\bf Step 1.}
Recalling Proposition~\ref{lipschitz_approx}, we have that for every $h$ sufficiently large there exists a  $1$-Lipschitz function $f_h:\R^{n-1}\to\R$ such that 
\begin{equation}\label{flatness4}
\sup_{\R^{n-1}}|f_h|\leq\eps_h^{\frac{1}{2(n-1)}},\qquad \Ha^{n-1}(\partial E_h\Delta\Gamma_{f_h}\cap B_{1/2})\leq C_3\eps_h,\qquad\int_{B^{n-1}_{1/2}}|Df_h|^2\,dx'\leq C_3\eps_h.
\end{equation}
Therefore, setting
$$
g_h(x'):=\frac{f_h(x')-a_h}{\sqrt{\eps_h}},\qquad \text{where\,\,\,} a_h= \medint_{B_{1/2}^{n-1}}f_h\,dx',
$$
we may assume, up to a not relabelled subsequence, that the functions $g_h$ converge weakly in $H^1(B^{n-1}_{1/2})$ and strongly in $L^2(B^{n-1}_{1/2})$ to a function $g$.

We claim that $g$ is harmonic in $B_{1/2}$. To prove this it is enough to show that for any $\varphi\in C^1_0(B^{n-1}_{1/2})$
\begin{equation}\label{flatness5}
\lim_{h\to\infty}\frac{1}{\sqrt{\eps_h}}\int_{B^{n-1}_{1/2}}Df_h\cdot D\varphi\,dx'=0.
\end{equation}
In order to prove this equality we fix $\delta>0$ so that supp$\,\varphi\times[-2\delta,2\delta]\subset B_{1/2}$, choose a cut-off function $\psi:\R\to[0,1]$ with support in $(-2\delta,2\delta)$ and $\psi\equiv1$ in $[-\delta,\delta]$, and apply to $E_h$ the weak Euler-Lagrange equation \eqref{euler21} with $X=(0,\dots,0,\varphi\psi)$. By the height bound (Lemma~\ref{height_bound})  for $h$ large it holds $\partial E_h\cap B_{1/2}\subset B^{n-1}_{1/2}\times(-\delta,\delta)$. Therefore by  denoting  $\nu'_{E_h}$ the vector made up by the first $n-1$ components of $\nu_{E_h}$, we  have
\[
\begin{split}
-\gamma\int_{\partial E_h\cap B_{1/2}}\nu_{E_h}\cdot e_nD\varphi\cdot\nu'_{E_h}\,d\Ha^{n-1}+\int_{B_{1/2}} \sigma_{E_h}\bigl(|Du_{h}|^2\varphi\psi'- 2D_nu_{h}Du_{h}\cdot D(\varphi\psi)\bigr) \, dx \\
\leq \Lambda r_h\int_{\partial E_h\cap B_{1/2}} |\varphi\psi| \, d \Ha^{n-1}\,.
\end{split}
\]
Thus, using the energy upper bound and recalling the inequality in \eqref{flatness2}, we have
$$
-\gamma\int_{\partial E_h\cap B_{1/2}}\nu_{E_h}\cdot e_nD\varphi\cdot\nu'_{E_h}\,d\Ha^{n-1}\leq C\eps_h
$$
for some constant $C$ depending on $\Lambda$, $\varphi$ and $\psi$, but independent of $h$. Therefore, dividing the above inequality by $\sqrt{\eps_h}$, letting $h\to\infty$ and replacing $\varphi$ by $-\varphi$ we conclude that
\[
\lim_{h\to\infty}\frac{1}{\sqrt{\eps_h}}\int_{\partial E_h\cap B_{1/2}}\nu_{E_h}\cdot e_nD\varphi\cdot\nu'_{E_h}\,d\Ha^{n-1}=0.
\]
From this equation we  get \eqref{flatness5} by observing that
\begin{align*}
&-\int_{\partial E_h\cap B_{1/2}}\nu_{E_h}\cdot e_nD\varphi\cdot\nu'_{E_h}\,d\Ha^{n-1}=-\int_{\Gamma_{f_h}\cap B_{1/2}}\nu_{E_h}\cdot e_nD\varphi\cdot\nu'_{E_h}\,d\Ha^{n-1}\\
&\qquad\qquad -\int_{(\partial E_h\setminus\Gamma_{f_h}) \cap B_{1/2}}\nu_{E_h}\cdot e_nD\varphi\cdot\nu'_{E_h}\,d\Ha^{n-1}+\int_{(\Gamma_{f_h}\setminus\partial E_h) \cap B_{1/2}}\nu_{E_h}\cdot e_nD\varphi\cdot\nu'_{E_h}\,d\Ha^{n-1}.
\end{align*}
Indeed, recalling the second inequality in \eqref{flatness4}, we have that
$$
0=\lim_{h\to\infty}\frac{-1}{\sqrt{\eps_h}}\int_{\Gamma_{f_h}\cap B_{1/2}}\nu_{E_h}\cdot e_nD\varphi\cdot\nu'_{E_h}\,d\Ha^{n-1}=\lim_{h\to\infty}\frac{1}{\sqrt{\eps_h}}\int_{B^{n-1}_{1/2}}\frac{Df_h\cdot D\varphi}{\sqrt{1+|Df_h|^2}}\,dx',
$$
from which \eqref{flatness5} immediately follows using the third inequality in \eqref{flatness4}.
\par\noindent {\bf Step 2.} Since $g$ is harmonic we have for $\tau\in(0,1/8)$ that 
\[
\begin{split}
\int_{B_{2\tau}^{n-1}}|g(x')- g(0)- Dg(0)\cdot x'|^2\,dx'&\leq c(n)\tau^{n+3} \sup_{B_{1/4}}|D^2g|^2\\
&\leq c(n)\tau^{n+3} \int_{B_{1/2}^{n-1}}|Dg|^2\,dx'.
\end{split}
\]
 Since by  \eqref{flatness4} we have that
$$
\int_{B_{1/2}^{n-1}}|Dg|^2\,dx\leq\liminf_{h\to\infty}\int_{B_{1/2}^{n-1}}|Dg_h|^2\,dx\leq C_3,
$$
and by the mean value property  $(g)_r:=\medintdue_{B_r^{n-1}}g dx'= g(0)$  and $(Dg)_r= Dg(0)$, we may conclude that
$$
\lim_{h\to\infty}\int_{B_{2\tau}^{n-1}}|g_h(x')-(g_h)_{2\tau}-(Dg_h)_{2\tau}\cdot x'|^2\,dx'=\int_{B_{2\tau}^{n-1}}|g(x')-(g)_{2\tau}-(Dg)_{2\tau}\cdot x'|^2\,dx'\leq \widehat C\tau^{n+3}\,,
$$
for some constant $\widehat C$ depending only on $C_3$ and $n$. In turn, recalling the definition of $g_h$, this inequality is equivalent to
$$
\lim_{h\to\infty}\frac{1}{\eps_h}\int_{B_{2\tau}^{n-1}}|f_h(x')-(f_h)_{2\tau}-(Df_h)_{2\tau}\cdot x'|^2\,dx'\leq \widehat C\tau^{n+3}.
$$
From this inequality, recalling that $|Df_h|\leq1$  and setting 
$$
c_h:=\frac{(f_h)_{2\tau}}{\sqrt{1+|(Df_h)_{2\tau}|^2}},\qquad\nu_h=\frac{(-(Df_h)_{2\tau},1)}{\sqrt{1+|(Df_h)_{2\tau}|^2}},
$$
we easily have
\begin{align*}
&\limsup_{h\to\infty}\frac{1}{\eps_h}\int_{\partial E_h\cap\Gamma_{f_h}\cap B_{2\tau}}|\nu_h\cdot x-c_h|^2\,d\Ha^{n-1}\\
&\qquad\qquad\qquad\leq\lim_{h\to\infty}\frac{\sqrt{2}}{\eps_h}\int_{B_{2\tau}^{n-1}}|f_h(x')-(f_h)_{2\tau}-(Df_h)_{2\tau}\cdot x'|^2\,dx'\leq\sqrt{2}\widehat C\tau^{n+3}.
\end{align*}
On the other hand, arguing as in Step 1, we immediately get from the height bound  and from the first two inequalities in \eqref{flatness4} that 
$$
\lim_{h\to\infty}\frac{1}{\eps_h}\int_{(\partial E_h\setminus\Gamma_{f_h})\cap B_{2\tau}}|\nu_h\cdot x-c_h|^2\,d\Ha^{n-1}=0.
$$
Hence, we conclude that
\begin{equation} \label{contradiction1}
\limsup_{h\to\infty}\frac{1}{\eps_h}\int_{\partial E_h\cap B_{2\tau}}|\nu_h\cdot x-c_h|^2\,d\Ha^{n-1}
\leq\sqrt{2}\widehat C\tau^{n+3}.
\end{equation}

Note that 
\[
\begin{split}
\int_{\partial E_h \cap B_{2 \tau}} |\nu_{E_h} - \nu_h|^2\, d \Ha^{n-1} &\leq  2\int_{\partial E \cap B_{2 \tau}} |\nu_{E_h} - e_n |^2\, d \Ha^{n-1} + 2 |e_n - \nu_h|^2 \Ha^{n-1}(\partial E_h \cap B_{2 \tau})\\
&\leq 4 \eps_h+ C \int_{B_{1/2}^{n-1}} |Df_h|^2 \, dx' \leq C \eps_h
\end{split}
\]
by the third inequality from \eqref{flatness4}. In particular, this shows that $\E_h(0, 2 \tau, \nu_h) \to 0 $ as $h \to \infty$. Therefore applying Theorem \ref{tilt_lemma} and \eqref{contradiction1} we have for $h$ large  that 
\[
\begin{split}
\E_h(0,\tau) \leq \E_h(0,\tau,\nu_h) \leq C_4(\hat{C} \tau^2 \E_h(0,1) + D_h(0,2\tau) + 2\tau r_h )
\end{split} 
\]
which is  a contradiction to \eqref{flatness3} if we choose $C_5 > C_4 \max\{\hat{C},2\}$.
\end{proof}

Finally we give the proof of the regularity theorem (Theorem \ref{mainthm1}).

\begin{proof}[Proof of Theorem \ref{mainthm1}]
\par\noindent {\bf Step 1.} We begin by proving that for every $\tau \in (0,1)$ there exists $\eps_6 = \eps_6(\tau) >0$ such that if $\E(x,r) \leq \eps_6$ then 
\[
\D(x,\tau r) \leq C_0 \tau \D(x, r)
\]
where $C_0$  is from Proposition \ref{decay_est}. We argue by contradiction. After performing the same translation and rescaling used in the proof of Lemma~\ref{decay_energy} we may assume that there exist  sequences $\eps_h, r_h>0$ and a sequence $E_h$ of $\Lambda r_h$-minimizers of $\F$ in $B_1$ with equibounded energies, such that, denoting by $\E_h$ the excess of $E_h$ and by $\D_h$ the rescaled Dirichlet integral of $u_{E_h}$, we have that $0 \in \partial E_h$,
\begin{equation} \label{contradiction_main}
\E_h(0,1) = \eps_h \to 0 \qquad \text{and} \qquad \D_h(0,\tau) > C_0 \tau\D_h(0,1).
\end{equation}
By the energy upper bound (Theorem \ref{energy_upper}) and the compactness lemma (Lemma \ref{compactness})  we may assume that $E_h \to E$ in $L^1(B_1)$ and $0 \in \partial E$. Moreover by the lower semicontinuity of the excess  we have that $\E(0,1) = 0$ where $\E(0,1)$ is the excess of $E$ at $0$. Thus it follows that $E$ is a half space, say $H$, in $B_1$, see \cite[Proposition 22.2]{M}. In particular, for $h$ large it holds
\[
|(E_h \Delta H) \cap B_1| < \eps_0(\tau) |B_1|
\]
where $\eps_0$ is from Proposition \ref{decay_est} which  gives a contradiction with the inequality in  \eqref{contradiction_main}.

\par\noindent {\bf Step 2.} Let $U \subset \! \subset \Omega$ be an open set. We show that for every $\tau \in (0,1)$ there exists $\eps = \eps(\tau, U)>0$ such that  if $x_0 \in \partial E$, $B_r(x_0)\subset U$ and  $\E(x_0,r) + \D(x_0,r) +r < \eps$ then 
\begin{equation}
\label{main_step2}
\E(x_0,\tau r) + \D(x_0,\tau  r) +\tau r \leq  C_6\tau (\E(x_0,r) + \D(x_0,r) +r). 
\end{equation}
First of all, if $\eps \leq \eps_6(\tau)$, then Step 1 implies 
\begin{equation} \label{main_step_2.2}
\D(x_0, \tau r) \leq  C_0 \tau\D(x_0,r).
\end{equation}
In order to prove \eqref{main_step2} we may assume that $\tau <1/2$.  Assume first  that  $D(x_0,r) + r \leq  \tau^{-n} \E(x_0,r)$. Then if  $ \E(x_0,r) < \min\{\eps_5(\tau,M), \eps_6(2\tau)\}$, for $M= \tau^{-n}$, it follows from Proposition \ref{flatness_impro} that 
\[
\begin{split}
\E(x_0,\tau r) &\leq C_5 (\tau^2 \E(x_0,r) + \D(x_0,2\tau r) + \tau r)\\
&\leq  C_5 (\tau^2 \E(x_0,r) + 2C_0 \tau \D(x_0, r) + \tau r)
\end{split}
\]
where the last inequality follows from \eqref{main_step_2.2} applied to $2\tau$. On the other hand if $\E(x_0,r) \leq \tau^n( D(x_0,r) + r)$ we immediately obtain
\[
\E(x_0,\tau r) \leq \tau^{1-n}  \E(x_0,r) \leq \tau( D(x_0,r) + r).
\]
Therefore \eqref{main_step_2.2} implies  \eqref{main_step2} by choosing $\eps = \min\{\eps_5(\tau,M), \eps_6(2\tau), \eps_6(\tau)\}$.

\par\noindent {\bf Step 3.} Let us fix $\sigma \in (0,1/2)$. We choose $\tau_0 \in (0,1)$ such that $C_6\tau_0 \leq \tau_0^{2\sigma}$ where $C_6$ is the constant in   \eqref{main_step2}. Let $U \subset \! \subset \Omega$ be an open set. We define
\[
\Gamma \cap U = \{ x \in \partial E \cap U \,: \, \E(x,r) + \D(x,r) +r < \eps(\tau_0, U) \,\,\text{for some }\, r>0 \,\, \text{such that  } \, B_r(x)\subset U\}
\]
where $\eps(\tau_0)$ is from Step 2. Note that $\Gamma \cap U$ is a relatively open in $\partial E$. We show that  $\Gamma \cap U$ is $C^{1,\sigma}$-hypersurface.

Indeed \eqref{main_step2} implies  via  standard iteration argument that  if $x_0 \in \Gamma \cap U$ there exist  $r_0>0$ and a neighborhood $V$ of $x_0$ such that for every $x \in \partial E \cap V$ it holds
\[
 \E(x,\tau_0^k r_0)+ \D(x,\tau_0^k r_0) + \tau_0^k  r_0 \leq \tau_0^{2 \sigma k} \qquad \text{for }\, k = 0,1,2,\dots.
\]
In particular
\[
 \E(x,\tau_0^k r_0)\leq \tau_0^{2 \sigma k}.
\]
From this estimate and the density lower bound, arguing exactly as in  \cite[Theorem 8.2]{Giusti2}, we obtain that for every $x \in \partial E \cap V$ and $0< s < t < r_0$ it holds
\begin{equation} \label{from_giusti}
|(\nu_E)_s(x)- (\nu_E)_t(x)| \leq c t^\sigma 
\end{equation}
for a constant $c$ depending on $\tau_0, r_0$ and $n$. Here  
\[
(\nu_E)_t(x) = \medint_{\partial E \cap B_t(x)} \nu_E(y)\, d \Ha^{n-1}.
\]
The estimate \eqref{from_giusti} first implies that $\Gamma \cap U$ is $C^1$ (see for instance \cite[Theorem 8.4]{Giusti2}). By a standard argument we then deduce again from \eqref{from_giusti} that  $\Gamma \cap U$ is $C^{1,\sigma}$-hypersurface. Finally  we define $\Gamma := \cup_i \, (\Gamma \cap U_i)$ where $(U_i)$ is an increasing sequence of open sets such that $U_i \subset\! \subset \Omega$ and $\Omega = \cup_i  U_i$

\par\noindent {\bf Step 4.}  Finally we prove that  exists $\eta>0$, depending only on $\frac{\beta}{\alpha}$ and $n$, such that 
\[
\Ha^{n-1-\eta}(\partial E \setminus \Gamma) = 0.
\]
Since the argument is fairly standard we only give the sketch of the proof,  see e.g. Section 5 in \cite{AFH},  \cite{DFR} and \cite{DF}.  We set 
\[
\Sigma = \{ x \in \partial E \setminus \Gamma\, : \, \lim_{r \to 0} D(x,r) =0 \}. 
\] 
Since by Lemma \ref{gehring} $Du \in L_{loc}^{2p}(\Omega)$ for some $p>1$, depending only on $\frac{\beta}{\alpha}$ and $n$, we have that 
\[
\dim_{\Ha}\Big( \{ x \in \Omega \, : \, \limsup_{r \to 0} D(x,r) >0 \}\Big)\leq n-p,
\]
where $\dim_{\Ha}$ denotes the Hausdorff dimension. The conclusion will follow if we show that $ \Sigma = \emptyset$ when $n\leq 7$ and $\dim_{\Ha}( \Sigma) \leq n-8$ otherwise.

Let us first treat the case $n \leq 7$. We argue by contradiction and assume, up to a translation, that $0 \in \Sigma$. Let us take a sequence $r_h \to 0$ and set $E_h = \frac{E}{r_h}$ and $u_h(x) = r_h^{-1/2} u(r_h x)$. Then $(E_h,u_h)$ is $\Lambda r_h$-minimizer of $\F$. Since $Du_h \to 0$ in $L^2(B_1)$  Lemma \ref{compactness} implies that, up to a subsequence, $E_h$ converges to a minimizer of the perimeter $E_\infty$ and moreover $\lim_{h \to \infty} P(E_h,U) = P(E_\infty,U)$ for every open set $U \subset B_1$ and $0 \in \partial E_\infty$.  Since $n \leq 7$, we know that $\partial E_\infty$ is a smooth manifold. In particular, for any $\eps>0$  there exists $r>0$ such that $\E(0, E_\infty,r)<\eps$. However the above convergence of the perimeter implies that $\E(0,E_h, r_h) < \eps$ when $h$ is large enough. By the definition of $\Gamma$ this contradicts the fact that $0 \in \Sigma$. 

In the case  $n \geq 8$ we claim that if $s > n-8$ then it holds $\Ha^{s}(\Sigma) = 0$. The proof of this can be achieved arguing exactly as in the proof of  \cite[Theorem 5.6]{AFH}.
\end{proof}

\section{Appendix}

 We conclude by  going through the proof of Lemma \ref{gehring} and estimate all the relevant constants  in the statement.

\begin{proof}[Proof of Lemma \ref{gehring}]
Without loss of generality we may assume that $Q_{R}(x_0)$ is the unit cube $Q$. Denote  $d(x)= \dist (x, \partial Q)$ and define
\[
\mathcal{C}_k = \{ x \in  Q \mid \frac{3}{4} 2^{-k-1} \leq d(x) \leq \frac{3}{4} 2^{-k} \}.
\]
Each $\mathcal{C}_k$ can be divided into cubes of side $\frac{3}{4} 2^{-k-1}$. We call this collection $\mathcal{G}_k$. By Lemma \ref{caccioppoli} we have for $F(x)= d(x)^n |Du(x)|^2$ that  
\[
\medint_{P} F \, dx \leq  C_0 \, \left(\medint_{\tilde{P}} F^{m} \, dx \right)^{\frac{1}{m}},
\]
where $\tilde{P}$ is the concentric cube to $P \subset \mathcal{C}_k$ or $P \subset Q_{1/4}$, for a constant $C_0 = 4^n C$, where $C$ is the constant from Lemma \ref{caccioppoli}.

Denote next $\Phi_t = \{ x \in Q \mid F(x)> t\}$, where $t > a := \medintdue_{Q} |Du|^2\, dx$.  Applying Calder\'on-Zygmund decomposition we obtain (in the proof \cite[Lemma 6.2]{Giusti} choose $\lambda = 2^{1/m} C_0$)
\[
\int_{\Phi_t} F \, dx \leq C_1 t^{1-m} \int_{\Phi_t} F^m \, dx
\]
for
\[
C_1 = 5^n 2^n \lambda =  2^{1/m} 10^n C_0 \leq 4\cdot  40^n \, C = C_{S,n}^2 \, 2^{10}\cdot  80^n  \frac{\beta}{\alpha}.
\]

The result of \cite[Proposition 6.1]{Giusti} now follows with the constants $A= C_1$ and $r= p>1$ such that 
\[
C_1(p-1) = \frac{p-m}{2}
\] 
that is
\[
p = \frac{2C_1 - m}{2C_1 -1}.
\]
This leads to the  inequality 
\[
\int_Q F^p \, dx \leq 2 a^{p-1} \int_Q F\, dx
\]
for $a = \medintdue_Q |Du|^2\, dx$. Recalling the definition of $F$ we finally obtain
\begin{equation}
\label{cubes}
 \medint_{Q_{1/2}}|Du|^{2p} \, dx \leq 2^{n+ pn+1}  \left( \medint_{Q}|Du|^2 \, dx \right)^p .
\end{equation}

Let $B_1 \subset \! \subset \Omega$. Observe that for any integer $h>1$,   $Q_{1/2}$ can be covered by $h^n$ cubes of side length  $1/h$. Hence, $B_{1/2}$ can be covered by $N_h$ cubes $Q_{1/2h}(x_i)$ having non-empty intersection with $B_{1/2}$ and  $N_h \leq h^n$. Using the rescaled analogue of the  inequality \eqref{cubes} we get
\[
\begin{split}
 \medint_{B_{1/2}}|Du|^{2p} \, dx &\leq \frac{2^n }{\omega_nh^n} \sum_{i=1}^{N_h}  \medint_{Q_{1/2h}(x_i)}|Du|^{2p} \, dx \\
&\leq  2^{n+ pn+1} \frac{2^n }{\omega_nh^n}  \sum_{i=1}^{N_h}  \left( \medint_{Q_{1/h}(x_i)}|Du|^{2} \, dx \right)^p \\
&\leq 2^{n+ pn+1} \frac{2^n }{\omega_nh^n}  \sum_{i=1}^{N_h}  \left( \frac{h^n \omega_n}{2^n}\medint_{B_1}|Du|^{2} \, dx \right)^p 
\end{split}
\]
provided $h > 4 \sqrt{n}$, in which case $Q_{1/h}(x_i) \subset B_1$ for every $i=1, \dots, N_h$. We may choose $h \leq 5 \sqrt{n}$ and thus we get
\[
 \medint_{B_{1/2}}|Du|^{2p} \, dx \leq   2^{2n+1} 5^{np} n^{np/2} \omega_n^{p-1} \left(\medint_{B_1}|Du|^{2} \, dx \right)^p. 
\]
\end{proof}

\section*{Acknowledgment}
The research of NF was partially supported by the 2008 ERC Advanced Grant 226234 ``Analytic Techniques for
Geometric and Functional Inequalities'' and by the 2013 FiDiPro project ``Quantitative Isoperimetric Inequalities''. The research of VJ was  supported by the  Academy of Finland grant 268393. The authors want also to thank R. Kohn for pointing out to them the phenomenon of Taylor cones and G. De Philippis and A. Figalli for finding a mistake in the first version of the paper.

\end{document}